\documentclass[10pt, leqno]{amsart}

\usepackage{amsmath}
\usepackage{amssymb}
\usepackage{amsthm}
\usepackage{mathrsfs}
\usepackage[dvips]{graphicx}
\usepackage{psfrag}
\usepackage[hang,small,bf]{caption}
\usepackage{amsaddr}

\newtheorem{theorem}{Theorem}
\newtheorem{corollary}[theorem]{Corollary}

\newtheorem{proposition}[theorem]{Proposition}
\newtheorem{lemma}[theorem]{Lemma}

\newtheorem{remark}[theorem]{Remark}

\numberwithin{equation}{section}
\numberwithin{theorem}{section}

\newcommand{\ent}{{H}}
\newcommand{\capacity}{{\textup{cap}}}
\newcommand{\size}{\langle  R_n \rangle_o}
\newcommand{\sizebd}{\langle  \partial_i R_n \rangle_o}
\newcommand{\ut}{\alpha}
\newcommand{\entt}{{H(R_n)}}
\newcommand{\sizet}{\langle  R_n \rangle}

\begin{document}

\title{\textbf{Entropy of random walk range on uniformly transient and on uniformly recurrent graphs}}
\author{David Windisch}
\address{The Weizmann Institute of Science\\
Faculty of Mathematics and Computer Science\\
POB 26\\
Rehovot 76100 \\
Israel\\
\vspace{2mm}
\textup{\texttt{david.windisch@weizmann.ac.il}}}
\date{}

\begin{abstract}
We study the entropy of the distribution of the set $R_n$ of vertices visited by a simple random walk on a graph with bounded degrees in its first $n$ steps. It is shown that this quantity grows linearly in the expected size of $R_n$ if the graph is uniformly transient, and sublinearly in  the expected size if the graph is uniformly recurrent with subexponential volume growth. This in particular answers a question asked by Benjamini, Kozma, Yadin and Yehudayoff \cite{BKYY09}. 
\end{abstract}

\maketitle

\section{Introduction}

Let $(X_n)_{n \geq 0}$ be a simple random walk on a graph $G=(V,E)$ starting at some vertex $o \in V$. The entropy of the distribution of $X_n$ for large $n$ has been studied on Cayley graphs and is known to be related to other objects of interest, such as the rate of escape and the existence of non-trivial bounded harmonic functions, see \cite{E03}, \cite{KV83}, \cite{V85}. This work is devoted to the entropy of a similar observable, to that of the range of the random walk. Let $R_n = \{X_0, X_1, \ldots, X_n\}$ be the set of vertices visited by the random walk in its first $n$ steps. The entropy of $R_n$ is defined as 
\begin{align}
\label{e:ent}
\ent_o(R_n) = E_o \Biggl[ \log \biggl( \frac{1}{p_o(R_n)} \biggr) \Biggr],
\end{align}
where the random walk starts at $o \in V$, $P_o$-a.s., $p_o(A) = P_o[R_n=A]$ for any finite connected set $A \subseteq V$ containing $o$, and $0 \log (1/0)$ is defined to be $0$ as usual. According to Shannon's noiseless coding theorem \cite{S48}, $\ent_o(R_n)/ \log 2$ can be interpreted as the approximate number of $0$-$1$-bits required per realization in order to encode a large number of independent realizations of $R_n$ with negligible probability of error. Up to an additive constant, $\ent_o(R_n)/\log 2$ can also be viewed as the expected number of bits necessary and sufficient to encode $R_n$ (see \cite{CT91}, Theorem~5.4.1), and as the expected number of fair coin tosses required for the simulation of $R_n$ (see \cite{CT91}, Theorem~5.12.3).

\medskip

The entropy of $R_n$ for random walk on ${\mathbb Z}^d$, $d \geq 1$, is investigated in a recent work by Benjamini, Kozma, Yadin and Yehudayoff \cite{BKYY09}. There, the authors show that the large $n$ behavior of $\ent_o(R_n)$ is of order $n$ for $d \geq 3$, of order $n/\log^2 n$ for $d=2$ and of order $\log n$ for $d=1$. Comparing this behavior with that of the expected size $\size$ of $R_n$,
\begin{align}
\label{e:size}
\size=E_o[|R_n|],
\end{align}
known to be of order $n$ for $d \geq 3$, of order $n/\log n$ for $d=2$ and of order $\sqrt{n}$ for $d=1$ (cf.~\cite{R05}, p.~221) , we observe that on ${\mathbb Z}^d$, $\ent_o(R_n)$ grows linearly in $\size$ in the transient case and only sublinearly in the recurrent case. On general graphs with bounded degrees, $\ent_o (R_n)$ can grow at most linearly in $\size$ (see Proposition~\ref{p:general}), but the assumptions on recurrence and transience alone do not allow us to conclusively compare $\ent_o(R_n)$ with $\size$. Under slightly stronger assumptions, however, we can generalize the observation just made for ${\mathbb Z}^d$ to large classes of graphs.

\medskip

The graphs we consider in this work are characterized by strengthened transience and recurrence conditions. Recall that a graph is transient if the escape probability $P_o[ o \notin \{X_1, X_2, \ldots \}]$ is strictly positive for any starting vertex $o \in V$. We call a graph $G$ \emph{uniformly transient} if such an estimate holds uniformly in $o$, that is, if
\begin{align}
\label{e:ut}
\inf_{o \in V} P_o[ o \notin \{X_1, X_2, \ldots \}] > 0.
\end{align}
Theorem~\ref{t:entlb} shows that $\ent_o(R_n)$ grows linearly in $\size$ on uniformly transient graphs with bounded degrees. Such a statement does not hold under the assumption of transience alone, see Remark~\ref{r:extrans} for a counterexample.
\begin{theorem} \label{t:entlb}
Let $G$ be a uniformly transient graph with bounded degrees. Then
\begin{align}
\label{e:entlb} \liminf_{n \to \infty} \inf_{o \in V} \frac{\ent_o(R_n)}{\size} > 0.
\end{align}
\end{theorem}
On the other hand, we call a graph $G$ \emph{uniformly recurrent} if
\begin{align}
\label{e:unire}
\sup_{o \in V} P_o [o \notin \{X_1, \ldots, X_n\}] \longrightarrow 0, \text{ as } n \to \infty.
\end{align}
In addition, let $|\partial_e B(x,r)|$ be the number of vertices at distance $r+1$ from $x \in V$ with respect to the usual graph distance. If the degrees of the graph are bounded by $d\geq 2$, then $|\partial_e B(x,r)| \leq d^{r+1}$. The following condition requires slightly more:
\begin{align}
\label{e:vg} \text{for any $\epsilon > 0$, } \sup_{x \in V} |\partial_e B(x,r)| \leq e^{\epsilon r}, \text{ for infinitely many } r \in {\mathbb N}.
\end{align} 
Under these conditions, we prove that $\ent_o(R_n)$ grows only sublinearly in $\size$. We note that recurrence alone does not imply such a statement, see Remark~\ref{r:exrec} and also Remark~\ref{r:ign}.
\begin{theorem} \label{t:entub}
Let $G$ be any uniformly recurrent graph with bounded degrees satisfying \eqref{e:vg}. Then
\begin{align}
\label{e:entub}
\sup_{o \in V} \frac{\ent_o(R_n)}{\size} \longrightarrow 0, \text{ as } n \to \infty.
\end{align}
\end{theorem}

\medskip

The above results apply in particular to all vertex-transitive graphs. Recall that a graph $G$ is \emph{vertex-transitive}, if for every pair of vertices $(x,x')$, there is a bijection $\phi$ from $V$ to $V$ such that $\phi(x)=x'$ and $d(y,y') = d(\phi(y),\phi(y'))$ for all $y, y' \in V$, where $d(.,.)$ denotes the usual graph distance. For vertex-transitive graphs, $\ent_o(R_n)$ and $\size$ do not depend on the starting vertex $o$ of the random walk, so we omit $o$ from the notation. We can deduce the following dichotomy from the results above:
\begin{theorem} \label{t:dich}
Let $G$ be a vertex-transitive graph. 
\begin{align}
\label{e:dich1}
&\text{If } G \text{ is transient} \text{, then }  \liminf_{n \to \infty} \frac{\entt}{\sizet} > 0,\\
\label{e:dich2}
&\text{if } G \text{ is recurrent} \text{, then } \frac{\entt}{\sizet} \longrightarrow 0, \text{ as } n \to \infty.
\end{align}
\end{theorem}

\medskip

For uniformly transient graphs, $\sizet$ grows linearly in $n$ (see Proposition~\ref{p:Rllnt}). Hence, \eqref{e:dich1} in particular shows that $\entt$ grows linearly in $n$ on all vertex-transitive and transient graphs, thus answering a question asked in \cite{BKYY09}, Section~3.3.

\medskip

We now comment on the proofs, starting with the one of Theorem~\ref{t:entlb}, given in Section~\ref{s:trans}. A simple observation made in \cite{BKYY09} shows that $\ent_o(R_n)$ is bounded from below by a constant times the expected size of the interior boundary $\partial_i R_n$ of $R_n$. It thus suffices to prove a lower bound of order $\size$ on $E_o[|\partial_i R_n|]$, or, equivalently, to prove that the fraction of the visited vertices belonging to $\partial_i R_n$ is non-degenerate. Note that a vertex belongs to $\partial_i R_n$ if it is visited by the random walk, but at least one of its neighbors is not. Not all visited vertices are equally likely to belong to $\partial_i R_n$. Indeed, a vertex in the middle of a long path of vertices of degree $2$ is very unlikely to end up in the boundary of $R_n$, because the random walk typically returns many times and visits both of its neighbors before escaping from it. In order to avoid such situations, we observe that on uniformly transient graphs, vertices of degree at least $3$ exist within a bounded distance of every vertex (cf.~Lemma~\ref{l:d3}), and use uniform transience to prove that whenever the random walk visits such a vertex $x$, there is a non-degenerate probability that the walk then escapes from the ball $B(x,1)$ of radius $1$ around it and leaves $x$ in $\partial_i R_n$. This yields the lower bound on $E_o[|\partial_i R_n|]$ required to prove Theorem~\ref{t:entlb}. In order to show that transience alone is not a sufficient assumption for the conclusion of Theorem~\ref{t:entlb}, we consider a binary tree, where every edge at depth $l$ is replaced by a path of length $l+3$, and show that this graph satisfies \eqref{e:entub}, see Remark~\ref{r:extrans} and Figure~\ref{fig:s}.

\medskip

In the proof of Theorem~\ref{t:entub} in Section~\ref{s:rec}, we cover every possible realization of $R_n$ with one of at most $\sup_{x \in V} |\partial_e B(x,r)|^{(const.)|R_n|/r}$ different collections of balls of radius $r \geq 1$. Uniform recurrence implies that most of these balls are typically completely covered by $R_n$. We then consider the conditional entropy of $R_n$, given the number $K$ of such balls intersected by $R_n$ and the number $L$ of balls that are not completely filled by $R_n$ (the definition of conditional entropy is recalled in \eqref{e:cent} below). The conditional entropy can then be bounded by the expected logarithm of the number of possible configurations $R_n$ can belong to, given $(K,L)$ (cf.~\eqref{e:cond}). By assumption~\eqref{e:vg}, the expected logarithm of the number of possible collections of covering balls can be made smaller than $\epsilon \size$ for any fixed $\epsilon>0$ by choosing $r$ appropriately, while the assumption of uniform recurrence yields a similar bound on the expected logarithm of the number of possible choices of the exact configurations of $R_n$ in the few unfilled balls. This argument proves the required estimate on the conditional entropy of $R_n$ given $(K,L)$. Since the entropy of $(K,L)$ itself is only of order at most $\log n$, this is sufficient for Theorem~\ref{t:entub}. Theorem~\ref{t:entub} cannot be proved for every recurrent graph. As a counterexample, we consider a sequence of finite binary trees with rapidly increasing depths, connected together by a half-infinite path, see Remark~\ref{r:exrec} and Figure~\ref{fig:c}. 
\medskip

The article is organized as follows: In Section~\ref{s:not}, we introduce notation and preliminary results on entropy, derive a general lower bound on $\size$ and show that $\ent_o(R_n)$ grows at most linearly in $\size$. Section~\ref{s:trans} contains the proof of Theorem~\ref{t:entlb} and Section~\ref{s:rec} proves Theorem~\ref{t:entub}. Finally, the dichotomy for vertex-transitive graphs in Theorem~\ref{t:dich} is deduced in Section~\ref{s:dich}.

\medskip

\textbf{Acknowledgement.} The author is grateful to Itai Benjamini for helpful conversations and to an anonymous referee for helpful remarks.

\section{Notation and preliminaries} \label{s:not}

In this section, we introduce the notation and prove some preliminary results. These include a lower bound of order $\sqrt{n}$ on the expected size of $R_n$ on a general infinite graph with bounded degrees, obtained in Proposition~\ref{p:rlb} by adapting an argument in \cite{NP08}, as well as the observation that $\ent_o(R_n)$ grows at most as fast as the expected size $\size$ on any infinite connected graph with bounded degrees. 

\medskip

Throughout this article, we let $G=(V,E)$ be a graph. $V$ denotes the set of vertices and $E$ the set of edges consisting of unordered pairs of vertices in $V$. Whenever $\{x,y\} \in E$, we write $x \sim y$ and call the vertices $x$ and $y$ neighbors. The number of neighbors of a vertex $x$ is always assumed to be finite and referred to as the degree of $x$, denoted $\deg(x)$.  A path of length $l$ is a sequence $(x_0, x_1, \ldots, x_l)$ of vertices in $V$ such that $x_i \sim x_{i+1}$ for $0 \leq i \leq l-1$. The distance $d(x,y)$ between any two vertices $x$ and $y$ is defined as the length of the shortest path starting at $x_0=x$ and ending at $x_{d(x,y)}=y$, and defined to be $\infty$ if no such path exists. $G$ is said to be connected if $d(x,y)< \infty$ for all $x,y \in V$. For any $x \in V$, $r \geq 0$, we define the ball $B(x,r) = \{y \in V: d(x,y) \leq r\}$. The graph $G$ has bounded degrees if $\sup_{x \in V} \deg(x) \leq d$ for some $d \geq 1$. For any set $A$ of vertices, we define the interior and exterior boundaries of $A$ by $\partial_i A = \{x \in A: x \sim y \text{ for some } y \in V \setminus A\}$ and $\partial_e A = \partial_i (V \setminus A)$. The subgraph $G(A)=(A,E(A))$ induced by $A \subseteq V$ consists of the vertices in $A$ and the set of edges $E(A) = \{\{x,y\} \in E: \{x,y\} \subseteq A\}$. We say that $A$ is connected if $G(A)$ is connected in the above sense. The cardinality of any set $A$ is denoted by $|A|$, the largest integer less than $a \in {\mathbb R}$ by $[a]$ and the minimum of numbers $a, b \in {\mathbb R}$ by $a \wedge b$. Throughout the text, $c$ or $c'$ are used to denote strictly positive constants with values changing from place to place. Dependence of constants on additional parameters appears in the notation. For example, $c_d$ denotes a positive constant possibly depending on $d$.

\medskip

For any $x \in V$, we denote by $P_x$ the law on $V^{\mathbb N}$ (equipped with the canonical $\sigma$-algebra generated by the coordinate projections) of the Markov chain on $V$ starting at $x$ with transition probability 
\begin{align*}
p(x,y) = \left\{ \begin{array}{ll} 1/\deg(x), &\text{if } x \sim y, \\
0, & \text{otherwise,}
\end{array} \right.
\end{align*}
and write $(X_n)_{n \geq 0}$ for the canonical coordinate process, referred to as the simple random walk on $G$. The corresponding expectation is denoted by $E_x$, the canonical shift-operators on $V^{\mathbb N}$ by $(\theta_n)_{n \geq 0}$. Note that $\deg(x) p(x,y) = \deg(y) p(y,x)$, meaning that the measure $\pi: A \mapsto \sum_{x \in A} \deg(x)$ on $V$ is a reversible measure for the simple random walk. The first entrance and hitting times of a set $A$ of vertices are defined as
\begin{align}
\label{e:hit}
\tau_A = \inf\{ n \geq 0: X_n \in A\}, \, \tau^+_A = \inf \{ n \geq 1: X_n \in A\},
\end{align}
where we write $\tau_x$ rather than $\tau_{\{x\}}$ if $A$ consists of a single element $x$. The capacity of a finite nonempty subset $A$ of $V$ is defined as
\begin{align}
\label{e:cap}
\capacity(A) = \sum_{x \in A} P_x[\tau^+_A = \infty] \deg(x).
\end{align}
We will generally refer to the starting vertex of the random walk as $o$. It will be convenient to define the collection ${\mathcal C}_o$ as
\begin{align}
\label{e:Co}
{\mathcal C}_o = \{A \subseteq V: A \text{ is finite, connected, and } o \in A\}.
\end{align}
Note that $R_n \in {\mathcal C}_o$, $P_o$-a.s. The entropy of the range $R_n$ of the random walk has been defined in \eqref{e:ent}. More generally, for any random variable $X$ taking values in a countable set $\mathcal X$, we define the entropy of $X$ by
\begin{align*}
\ent_o(X) = E_o \bigg[ \log \bigg( \frac{1}{p_o(X)} \bigg) \bigg], \text{ where } p_o(x) = P_o[X=x], \text{ for } x \in {\mathcal X}.
\end{align*}
Given another random variable $Y$ taking values in a countable set $\mathcal Y$, the conditional entropy of $X$ given $Y$ is defined by
\begin{align}
\label{e:cent}
\ent_o(X|Y) = \ent_o \big((X,Y) \big) - \ent_o (Y).
\end{align}
An application of Jensen's inequality shows that
\begin{align}
\label{e:entcomb}
\ent_o(X) \leq \log |{\mathcal X}|,
\end{align}
while the following estimate is elementary:
\begin{align}
\label{e:condest}
\ent_o(X) \leq \ent_o(X|Y) + \ent_o(Y).
\end{align}
Moreover, we have the following useful lemma:

\begin{lemma}
\label{l:cond}
For random variables $X$ and $Y$ with values in countable sets $\mathcal X$ and $\mathcal Y$, 
\begin{align}
\label{e:cond}
&\ent_o(X|Y) \leq E_o[ \log |{\mathcal X}_Y|], \text{ where}\\
&{\mathcal X}_y =  \{ x \in {\mathcal X} : P_o[(X,Y)=(x,y)]>0\} , \text{ for } y \in {\mathcal Y}. \nonumber
\end{align}
\end{lemma}

\begin{proof}
Using Jensen's inequality, we have
\begin{align*}
&\ent_o(X|Y) = \sum_{\substack{(x,y) \in {\mathcal X} \times {\mathcal Y}: \\ P_o[(X,Y)=(x,y)]>0}} P_o \big[ (X,Y) = (x,y) \big] \log \bigg( \frac{P_o [ Y= y ]}{P_o \big[ (X,Y) = (x,y) \big]} \bigg)  \\
&\quad= \sum_{y \in {\mathcal Y}} P_o[Y=y] \sum_{x \in {\mathcal X}_y} P_o[X=x|Y=y] \log \bigg( \frac{1}{P_o[X=x|Y=y]} \bigg)\\
&\quad \leq E_o[ \log |{\mathcal X}_Y|]. \qedhere
\end{align*}
\end{proof}

The following simple lemma, combined with a covering argument, will be instrumental in proving the bound on $\ent_o(R_n)$ in Theorem~\ref{t:entub}.

\begin{lemma} \label{l:path}
Let $G=(V,E)$ be any graph with bounded degrees, $o \in V$, and let $A \in {\mathcal C}_o$  \textup{(}cf.~\eqref{e:Co}\textup{)}. Then there exists a nearest-neighbor path in $G(A)$ starting at $o$, covering $A$, and of length at most $2|A|$. 
\end{lemma}

\begin{proof}
The standard depth-first search algorithm (see \cite{CLR01}) yields a spanning tree $T_A=(A,E_T)$ of $G(A)$ (i.e.~a tree with vertices $A$ and edges $E_T \subseteq E(A)$), as well as a nearest-neighbor path in $T_A$ covering $A$ and traversing every edge in $E_T$ at most once in every direction. Since the length of such a path is bounded from above by $2|E_T| = 2(|A|-1) < 2|A|$, the statement follows.
\end{proof}

We now prove a general lower bound on the expected range of random walk on an infinite connected graph with bounded degrees. This lemma will allow us to disregard small errors when relating $\ent_o(R_n)$ to $\size$.

\begin{proposition} \label{p:rlb}
Let $G=(V,E)$ be any infinite connected graph with bounded degrees. Then
\begin{align}
\label{e:rlb} \liminf_{n \to \infty} \inf_{o \in V} \frac{\size}{\sqrt{n}} > 0.
\end{align}
\end{proposition}

\begin{proof}
This proof is an adaptation of an argument appearing in \cite{NP08} in the context of random walk on Cayley graphs. Let $(N_k)_{k \geq 1}$ be the successive times when the random walk visits a new vertex, that is, $N_1=0$, and for $k \geq 1$, $N_k = \inf\{n > N_{k-1}: X_n \notin \{X_0, \ldots X_{N_{k-1}}\} \}$. Then we have for $k \geq 1$ and $o \in V$, by the Chebychev inequality,
\begin{align*}
P_o[|R_n| \leq k] &= P_o \Bigl[ \bigl| \bigl\{ 0 \leq l \leq n: X_l \in \{X_{N_1}, \ldots, X_{N_k} \} \bigr\} \bigr| = n+1 \Bigr] \\
& \leq \frac{1}{n+1} \sum_{1 \leq i \leq k} E_o \biggl[ \sum_{0 \leq l \leq n} \mathbf{1} \{X_l = X_{N_i}\} \biggr].
\end{align*}
Noting that $X_l \neq X_{N_i}$ for $l < N_i$ and applying the strong Markov property at time $N_i$, we deduce that
\begin{align*}
P_o[|R_n| \leq k] &\leq \frac{k}{n} \sum_{0 \leq l \leq n} \sup_{x \in V} P_x[X_l =x].
\end{align*}
By a general heat-kernel upper bound, we have $\sup_{x,y \in V}P_x[X_l=y] \leq c_d/\sqrt{l+1}$, for some constant $c_d>0$ depending on the uniform bound $d$ on degrees (see, for example, \cite{W00}, Corollary~14.6, p.~149). Hence, we find
\begin{align*}
P_o[|R_n| \leq k] &\leq c_d \frac{k}{\sqrt{n}}, \text{ for } k \geq 1.
\end{align*}
This last inequality, applied with $k = [\sqrt{n}/(2c_d)]$, yields 
\begin{equation*}
E_o[|R_n|] \geq \frac{\sqrt{n}}{2c_d} P_o[|R_n| > [\sqrt{n}/(2c_d)]] \geq \frac{\sqrt{n}}{4c_d}. \qedhere
\end{equation*} 
\end{proof}

Finally, we prove that on any infinite connected graph with bounded degrees, $\ent_o(R_n)$ does not grow faster than $\size$.

\begin{proposition}
\label{p:general}
Let $G=(V,E)$ be any infinite connected graph with degrees bounded by $d$. Then
\begin{align*}
\limsup_{n \to \infty} \sup_{o \in V} \frac{\ent_o(R_n)}{\size} \leq  2 \log d.
\end{align*}
\end{proposition}

\begin{proof}
Lemma~\ref{l:cond} and Lemma~\ref{l:path} together imply that
\begin{align*}
\ent_o(R_n||R_n|) \leq \size 2 \log d,
\end{align*}
whereas it follows from \eqref{e:entcomb} that
\begin{align*}
\ent_o(|R_n|) \leq \log (n+1).
\end{align*}
Hence by \eqref{e:condest}, $\ent_o(R_n) \leq \size 2 \log d + \log (n+1)$. Proposition~\ref{p:rlb} completes the proof.
\end{proof}

\section{The transient case} \label{s:trans}

In this section, we prove Theorem~\ref{t:entlb} asserting that the entropy of $R_n$ grows at least linearly in its expected size $\size$ on any uniformly transient graph. 

\medskip

Lemma~\ref{l:Rbd} and Corollary~\ref{c:bdry} show that $\ent_o(R_n)$ grows at least linearly in the expected size $\sizebd$ of $\partial_i R_n$. These two statements and their proofs are straightforward adaptations of Lemma~3 and Corollary~4 in \cite{BKYY09}.

\begin{lemma} \label{l:Rbd}
For any graph $G$ with degrees bounded by $d \geq 2$,
\begin{align}
\label{e:rbd} P_o[R_n=A] \leq (1- d^{-1})^{|\partial_i A|-1}, \text{ for all } o \in A \subseteq V.
\end{align}
\end{lemma}

\begin{proof}
Fix $A \subseteq V$ and define the successive entrance times to $\partial_i A$ by $T_1 = \tau_{\partial_i A}$ and for $k \geq 2$, $T_k = \inf \{n > T_{k-1}: X_n \in \partial_i A\}$. Then $\{R_n = A\} \subseteq \{T_{|\partial_i A|} \leq n\} \subseteq \{T_{|\partial_i A| -1} <n\}$. Note that on the event $\{ T_k< \infty\}$, we have $P_{X_{T_k}}[X_{T_k+1} \in A] \leq 1- (1/d)$. An inductive application of the strong Markov property at the times $T_{|\partial_i A|-1}, T_{|\partial_i A|-2}, \ldots, T_1$ therefore yields
\begin{equation*}
\begin{split}
P_o[R_n=A] &\leq P_o \biggl[ \bigcap_{1 \leq k \leq |\partial_i A|-1} \{T_k < n, X_{T_k +1} \in  A\} \biggr]\\
&\leq \Bigl( 1- \frac{1}{d} \Bigr)^{|\partial_i A|-1}. \qedhere
\end{split}
\end{equation*}
\end{proof}

\begin{corollary} \label{c:bdry}
For any graph $G$ with degrees bounded by $d \geq 2$,
\begin{align}
\label{e:bdry} \ent_o(R_n)  \geq (\sizebd -1 ) \log \bigl((1-d^{-1})^{-1} \bigr),
\end{align}
where $\sizebd = E_o[|\partial_i R_n|].$
\end{corollary}

\begin{proof}
By Lemma~\ref{l:Rbd}, 
\begin{equation*}
\ent_o(R_n) \geq \log \bigl((1-d^{-1})^{-1} \bigr) \sum_{A \in {\mathcal C}_o} P[R_n=A] \bigl( |\partial_i A| -1 \bigr). \qedhere
\end{equation*}
\end{proof}

By the last corollary, it suffices to control $\sizebd$ in order to prove a lower bound on $\ent_o(R_n)$. We will eventually prove that many vertices of degree at least $3$ typically belong to $\partial_i R_n$. To this end, we prove in the following lemma that such vertices exist within constant distance of any vertex in uniformly transient graphs. For notational convenience, we introduce the parameter
\begin{align}
\label{e:utpar}
\ut = \inf_{x \in V} P_x [\tau_x^+ = \infty],
\end{align}
which is strictly positive if $G$ is uniformly transient (cf.~\eqref{e:ut}).

\begin{lemma}
\label{l:d3}
Let $G$ be a uniformly transient graph with $\ut$ as in \eqref{e:utpar}. Then for any $x \in V$, there is a vertex $y$ with $\deg(y) \geq 3$ and $d(x,y) \leq 1/\alpha$.
\end{lemma}

\begin{proof}
Let $x \in V$, and let $r$ be the distance from $x$ to the vertex of degree at least $3$ closest to $x$. Then
\begin{align}
\label{d31}
\alpha \leq P_x[\tau^+_x = \infty] &\leq P_x[\tau_{\partial_i B(x,r)} < \tau^+_x]\\
&= \sum_{y: y \sim x} \frac{1}{\deg(x)} P_y[\tau_{\partial_i B(x,r)} < \tau_x]. \nonumber
\end{align}
Consider any neighbor $y$ of $x$ and let $I_y$ be the set of vertices in $B(x,r) \setminus \partial_i B(x,r)$ that are connected to $y$ in $(B(x,r) \setminus \partial_i B(x,r)) \setminus \{x\}$. Then $I_y$ is connected and consists of vertices of degree at most $2$. If no vertex in $I_y$ is a neighbor of a vertex in $\partial_i B(x,r)$, then the escape probability on the right-hand side of \eqref{d31} equals $0$. Otherwise, the escape probability is equal to the probability that a simple random walk on $\mathbb Z$ started at $1$ reaches $r$ before $0$, hence to $1/r$ (see, for example, \cite{D05}, Chapter 3, Example~1.5, p.~179). It follows that $\alpha \leq 1/r$. 
\end{proof}

We will also require the following simple consequence of monotonicity of the capacity (cf.~\eqref{e:cap}):

\begin{lemma} \label{l:esc}
Let $G$ be a uniformly transient graph. Then 
\begin{align}
\label{e:esc} \inf_{A \subseteq V, A \neq \emptyset} \capacity(A) \geq \ut.
\end{align}
\end{lemma}

\begin{proof}
By uniform transience, $\capacity(\{x\}) \geq \ut \deg(x) \geq \ut$ for all $x \in V$ (see \eqref{e:ut}). The statement \eqref{e:esc} thus follows immediately from monotonicity of $\capacity$:
\begin{align}
\label{e:capmon}
\text{for finite, non-empty sets } A \subseteq B \subseteq V, \, \capacity(A) \leq \capacity(B).
\end{align}
Here is a direct proof of this well-known property: by summing over all possible times $n$ and locations $y$ of the last visit of the random walk to the set $B$ and using the simple Markov property at time $n$, we find that, for $A, B$ as above,
\begin{align}
\label{esc1}
\capacity(A) &= \sum_{x \in \partial_i A} \sum_{n=0}^\infty \sum_{y \in \partial_i B} P_x [\tau_A^+ > n, X_n = y, \tau_B^+ \circ \theta_n = \infty ] \deg(x)\\
&= \sum_{x \in \partial_i A} \sum_{n=0}^\infty \sum_{y \in \partial_i B} P_x [\tau_A^+ > n, X_n = y] P_y[\tau_B^+ = \infty ] \deg(x). \nonumber
\end{align}
Reversibility of the simple random walk implies that
\begin{align*}
\deg(x) P_x[\tau_A^+ > n, X_n=y] = \deg(y) P_y [\tau_A = n, X_n=x],
\end{align*}
hence by \eqref{esc1} that
\begin{align*}
\capacity(A) &= \sum_{y \in \partial_i B} \deg(y) P_y[\tau_A<\infty] P_y[\tau_B^+ = \infty] \leq \capacity(B),
\end{align*}
proving \eqref{e:capmon}.
\end{proof}

We now come to the main objective of this section.

\begin{proof}[Proof of Theorem~\ref{t:entlb}.]
We will prove that 
\begin{align}
\label{entlb1}
\liminf_{n \to \infty} \inf_{o \in V} \frac{\sizebd}{\size} >0.
\end{align}
The desired statement then follows by Corollary~\ref{c:bdry} and the fact that $$\inf_{o \in V} \size \to \infty,$$ proved in Proposition~\ref{p:rlb} (indeed, all connected components of a transient graph are infinite). Denote the set of vertices with degree at least $3$ by $V_{\geq 3}$ and the uniform bound on the degrees by $d$. Observe that any vertex $x$ in $V_{\geq 3}$ belongs to $\partial_i R_n$ if it is visited by the random walk, but at least one of its neighbors is not. Hence, for any $x \in V_{\geq 3}$, the probability of $\{x \in \partial_i R_n\}$ is bounded from below by the probability that the random walk first reaches $B(x,1)$ at some vertex $y \sim x$, then follows the path $(y,x,z)$, where $z$ is the neighbor of $x$ maximizing $P_z [\tau_{B(x,1)}^+   = \infty]$, and does not return to $B(x,1)$ until time $n$. By the strong Markov property applied at time $\tau_{B(x,1)} + 2$ and at time $\tau_{B(x,1)}$, we hence obtain that for any $o \in V$,
\begin{align}
\label{entlb2}
&P_o[x \in \partial_i R_n] \geq \\
&P_o[\tau_{B(x,1)} \leq n-2, (X_1, X_2) \circ \theta_{\tau_{B(x,1)}} = (x,z), \tau^+_{B(x,1)} \circ \theta_{\tau_{B(x,1)+2}}  = \infty] \nonumber\\
&\geq P_o[\tau_{B(x,1)} \leq n-2] \frac{1}{d^2} \frac{\ut}{d^2}, \text{ for any $x \in V_{\geq 3}$,} \nonumber
\end{align}
where we have applied Lemma~\ref{l:esc} with $A=B(x,1)$ in order to bound the escape probability $P_z[\tau^+_{B(x,1)} = \infty]$ from below by $\alpha/d^2$. For any $r \geq 1$, the random walk reaches $B(x,1)$ if it enters $B(x,r)$ and then follows the shortest path to $B(x,1)$, hence,
\begin{align*}
P_o[\tau_{B(x,1)} \leq n-2] \geq P_o[\tau_{B(x,r)} \leq n-1-r] d^{-(r-1)}.
\end{align*}
Using this estimate in \eqref{entlb2}, we obtain
\begin{align}
\label{entlb3}
\sizebd &\geq \sum_{x \in V_{\geq 3}} P_o[x \in \partial_i R_n] \\
&\geq \sum_{x \in V_{\geq 3}} P_o[\tau_{B(x,r)} \leq n-1-r] \frac{\alpha}{d^{3+r}} \nonumber \end{align}
We now fix $r = [1/\alpha]+1$ and note that, by Lemma~\ref{l:d3}, $\cup_{x \in V_{\geq 3}} B(x,r) = V$. In particular, the random walk is in one of the balls $B(x,r)$, $x \in V_{\geq 3}$, at every step, and therefore visits at least $[|R_k|/\sup_x |B(x,r)|]$ of them in its first $k \geq 0$ steps. Hence, we can deduce from \eqref{entlb3} that
\begin{align*}
\sizebd &\geq \left( \frac{E_o[|R_{n-1-r}|]}{\sup_{x \in V}|B(x,r)|} - 1 \right) \frac{\ut}{d^{3+r}} \\
&\geq \left(\frac{\size-1-r}{d^{1+r}} - 1 \right) \frac{\ut}{d^{3+r}}, \text{ for any } o \in V.
\end{align*}
Using again that $\inf_{o \in V} \size \to \infty$ by Proposition~\ref{p:rlb}, we obtain \eqref{entlb1}.
\end{proof}

\begin{figure}
\psfrag{S1}[cc][cc][4][0]{$S_1$}
\psfrag{S2}[cc][cc][4][0]{$S_2$}
\psfrag{r}[cc][cc][4][0]{$o$}
\begin{center}
\includegraphics[angle=0,
width=0.6\textwidth]{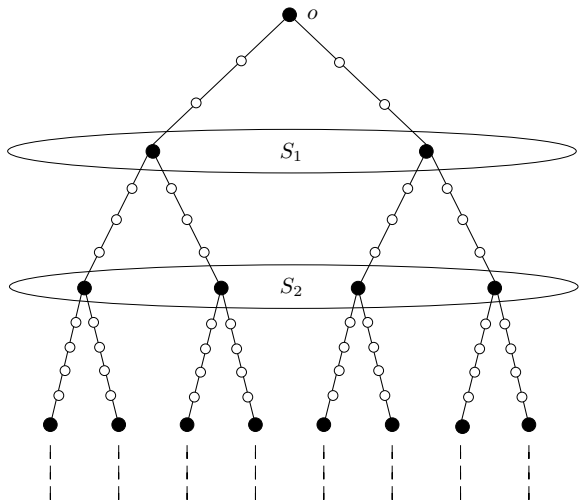}
\end{center}
\caption{The stretched binary tree constructed in Remark~\ref{r:extrans}. The filled vertices are the ones present in the original binary tree $T_b$, the stretched tree $G$ is obtained by adding the unfilled ones.} \label{fig:s}
\end{figure}
\begin{remark}
\label{r:extrans}
\textup{
The conclusion \eqref{e:entlb} of Theorem~\ref{t:entlb} does not hold for every transient graph with bounded degrees. For a counterexample, consider a binary tree $T_b=(V_b,E_b)$ and denote by $S_l = \{x \in V_b: d_{T_b}(o,x)=l\}$ the set of vertices at distance $l$ from the root. We now split every edge connecting a vertex in $S_{l-1}$ to a vertex in $S_l$ into a path of length $l+2$, for $l \geq 1$, and thereby obtain the stretched tree $G=(V,E)$, illustrated in Figure~\ref{fig:s}. Let now $(\sigma_n)_{n \geq 1}$ be the successive displacements of the random walk in $V_b$, that is, $\sigma_1 = \tau^+_{V_b}$ and for $i \geq 2$, $\sigma_i = \inf\{ n > \sigma_{i-1}: X_n \in \tau_{V_b} \}$. Then elementary computations using the simple Markov property at time $1$ (cf.~\cite{D05}, Chapter~4, Example~7.1, p.~271-272) show that for any vertex $x$ in $S_l$, 
\begin{align}
\label{cex0}
P_x[d_{T_b}(X_{\sigma_1},o)>l]= \frac{2}{2+(l+3)/(l+2)} \geq \frac{6}{10} > \frac{1}{2}, \text{ for } l \geq 1.
\end{align}
Hence, the process $(X_{\sigma_k})_{k \geq 1}$ is transient. In particular, the stretched graph $G$ remains transient. Denote by $R_n' = R_n \cap V_b$ the intersection of $R_n$ with the vertices in the original tree. Then from \eqref{e:condest}, we obtain that 
\begin{align}
\label{cex1}
\ent_o(R_n) &\leq \ent_o (R_n | R_n') + \ent_o (R_n'). 
\end{align}
For any fixed realization of $R_n'$ of diameter $m$ with respect to the original tree $T_b$, the set $R_n$ is determined up to the precise location of at most $2 |R_n'|$ boundary vertices on paths of length at most $m+3$. Hence, for any such realization, there are at most $(m+3)^{2|R_n'|} \leq (c |R_n|)^{2|R_n'|}$ choices for $R_n$. It therefore follows from Lemma~\ref{l:cond} that 
\begin{align*}
\ent_o(R_n | R_n') \leq c E_o[|R_n'|] \log n.
\end{align*}
The amount of time the process $X_{\sigma_.}$ spends in any given set $S_k$ is by \eqref{cex0} and an elementary estimate on biased random walk stochastically dominated by a geometrically distributed random variable with success probability $1/5$ (cf.~\cite{D05}, Chapter~4, Example~7.1 (b), p.~271-272). It follows that \begin{align}
\label{cex2}
E_o[|R_n'|] &\leq1+ 5 E_o \biggl[\sum_{1 \leq k \leq n} {\mathbf{1}}_{\{\tau_{S_k} \leq n\}} \biggr] \\
&\leq c E_o \Big[ \max_{0 \leq k \leq n} d(o,X_k)^{1/2} \Big] \leq c E_o [\sqrt{|R_n|}]. \nonumber
\end{align}
Applying Jensen's inequality on the right-hand side, we deduce that
\begin{align*}
\ent_o(R_n | R_n') \leq c \size^{1/2} \log n.
\end{align*}
Regarding the other term on the right-hand side of \eqref{cex1}, we have by the same argument as in the proof of Proposition~\ref{p:general}, 
\begin{align*}
\ent_o (R_n') \leq c E_o [| R_n'|] \leq c \size^{1/2}, \text{ by } \eqref{cex2}.
\end{align*}
Upon inserting the last two estimates into \eqref{cex1}, it follows from Proposition~\ref{p:rlb} that $\sup_{o \in V} \ent_o(R_n) / \size$ tends to $0$ as $n \to \infty$, although $G$ is transient.
}
\end{remark}


\section{The recurrent case} \label{s:rec}

In this section we prove Theorem~\ref{t:entub}, showing that $\sup_{o \in V} \ent_o(R_n)/\size$ tends to zero as $n \to \infty$ on uniformly recurrent graphs satisfying the growth condition \eqref{e:vg}. 

\medskip

Recall that any realization of $R_n$ belongs to ${\mathcal C}_o$, $P$-a.s. (cf.~\eqref{e:Co}). Our argument is based on Lemma~\ref{l:path} showing that any set $A \in {\mathcal C}_o$ can be explored by a path not longer than $2|A|$.  We will use this fact in the proof of Theorem~\ref{t:entub} in order to show that any realization of $R_n$ can be covered by one of at most $\sup_{x \in V} |\partial_e B(x,r)|^{d|R_n|/r}$ different collections of balls of radius $r$. The second key observation will be that due to uniform recurrence, most of these balls of radius $r$ visited are typically completely covered by the random walk. This will show that the distribution of $R_n$ is sufficiently concentrated to admit a bound on $\ent_o(R_n)$ sublinear in $\size$.

%
%
\begin{proof}[Proof of Theorem~\ref{t:entub}.]
We denote the uniform bound on degrees by $d$. Note that we can assume without loss of generality that $G$ is connected, for we may otherwise only consider the component of $G$ containing the starting vertex $o$ of the random walk. If $G$ is finite, then \eqref{e:entub} immediately follows from the fact that $\lim_n P_o[R_n = V]=1$. Hence, we can from now on assume that $G$ is connected and infinite. In particular, Proposition~\ref{p:rlb} is available for application.

Consider any $o \in V$ and $r \geq 1$. For any $A \in {\mathcal C}_o$ (cf.~\eqref{e:Co}), we define the finite sequence of vertices 
\begin{align}
\label{entub00} F_r(A)=(x_1, \ldots, x_k)
\end{align}
such that the balls with radius $r$ centered at $x_1, \ldots, x_k$ cover $A$, trying to keep $k$ as small as possible. We define $F_k(A)$ inductively as follows: by Lemma~\ref{l:path}, there is a nearest-neighbor path $p_A=(p(0), \ldots, p(l))$ in $G(A)$ starting at $o=p_A(0)$ and visiting all vertices in $A$ in $l \leq 2|A|$ steps. Set $x_1 = o$, $t_1=0$, and for $i \geq 2$ such that $t_{i-1} < \infty$, define $t_i$ as 
\begin{align*}
t_i = \left\{ \begin{array}{ll} \inf\{t > t_{i-1}: p_A(t) \notin B(x_{i-1},r)\}, & \text{if } \cup_{1 \leq j \leq i-1} B(x_j,r) \nsupseteq A, \\
\infty, & \text{otherwise,}
\end{array} \right.
\end{align*}
and, provided $t_i < \infty$, define $x_i = p_A(t_i)$. Since $p_A$ is of length at most $2|A|$ and covers all of $A$, this yields a finite sequence as in \eqref{entub00}, where $k$ is the largest index such that $t_k < \infty$ and satisfies
\begin{align}
\label{entub0}
k \leq 1+ [2|A|/r]. 
\end{align}
Finally, the construction implies that 
\begin{align}
\label{entub0.1}
\cup_{i=1}^k B(x_i,r) \supseteq A.
\end{align}
We now partition the collection ${\mathcal C}_o$ according to the cardinality of $F_r(A)$ and the number of vertices $x$ in $F_r(A)$ with the property that $B(x,r)$ is not completely filled by $A$: we have ${\mathcal C}_o = \cup_{k=1}^{\infty} \cup_{l=0}^k {\mathcal C}(k,l)$, where the disjoint collections ${\mathcal C}(k,l)$ are defined as
\begin{align}
\label{entub0.2}
{\mathcal C}(k,l) = \{A \in {{\mathcal C}_o} : |F_r(A)|=k, \, |\{x \in F_r(A): A \nsupseteq B(x,r)\}|=l\}. 
\end{align}
We introduce the random variables $K$ and $L$, defined as
\begin{align}
\label{entub0.2.00}
K = |F_r(R_n)|, \quad L = |\{x \in F_r(R_n): R_n \nsupseteq B(x,r)\}|.
\end{align}
Observe that since $|R_n| \leq n+1$, \eqref{entub0} implies that $$0 \leq L \leq K \leq [2(n+1)/r] + 1, \, P_o\text{-}a.s.$$
Using the elementary estimate \eqref{e:condest} together with \eqref{e:entcomb} and Lemma~\ref{l:cond}, we obtain the following bound on the entropy of $R_n$: 
\begin{align}
\label{entub0.2.0}
\ent_o(R_n) &\leq \ent_o \big( (K,L) \big) + \ent_o \big( R_n | (K,L) \big) \\
&\leq c_d \log n + E_o \big[ \log |{\mathcal C}(K,L)| \big]. \nonumber
\end{align}
In order to bound the cardinality of ${\mathcal C}(k,l)$ for $k \geq 1$, $l \geq 0$, we denote the maximal size of the ball and of the sphere of radius $r \geq 1$ by 
\begin{align}
\label{entub0.2.1}
b_r = \sup_{x \in V} |B(x,r)| \leq d^{1+r} \text{ and } s_r = \sup_{x \in V} |\partial_e B(x,r)|. 
\end{align}
We now bound the number of choices we have when selecting a set $A$ from the collection ${\mathcal C}(k,l)$. For a set $A$ to belong to ${\mathcal C}(k,l)$, $F_r(A)$ must consist of $k$ vertices. The initial vertex $x_1$ must be $o$ and each successive vertex must be at distance $r+1$ from the previous one. Hence, $F_r(A)$ can take at most $s_r^k$ different values. For every choice of $F_r(A) = (x_1, \ldots, x_k)$, there are at most $\binom{k}{l} \leq 2^k$ different choices of size-$l$ subsets of vertices $x_i$ with the property that $B(x_i,r)$ is not completely filled by $A$. After selecting $F_r(A)$ and such a subset, $A$ is by \eqref{entub0.1} determined up to the at most $(2^{b_r})^l$ possible configurations in the balls with centers $x_i$ that are not subsets of $A$. Hence, we have 
\begin{align}
\label{entub0.3}
\log |{\mathcal C}(k,l)| \leq k (\log s_r + \log 2) + b_r l \log 2. 
\end{align}
Inserting this estimate into the above bound \eqref{entub0.2.0} on $\ent_o(R_n)$, we obtain
\begin{align}
\label{entub1} \ent_o(R_n) \leq & E_o [K]  (\log s_r + \log 2) +  E_o [L] b_r \log 2 + c_d  \log n.
\end{align}
By \eqref{entub0}, $F_r(R_n)$ is of cardinality at most $(2|R_n|/r)+1$, hence
\begin{align}
\label{entub1.1}
E_o [K] \leq  \frac{2}{r} \size +1.
\end{align}
We now write $P(r,n) = \sup_{x \in V} P_x [R_{[n^{1/4}]} \nsupseteq B(x,r)]$ and bound the other expectation on the right-hand side of \eqref{entub1}. Since $F_r(R_n) \subseteq R_n$, 
\begin{align}
\label{entub1.2}
E_o [L] &\leq \sum_{x \in V} P_o[\tau_x \leq n, R_n \nsupseteq B(x,r)] \displaybreak[0] \\
&= \sum_{x \in V} \sum_{m=0}^n P_o[\tau_x=m] P_x[R_{n-m} \nsupseteq B(x,r)] \text{ (Markov prop.)} \nonumber\\
&\leq \sum_{x \in V} \Biggl( \sum_{m=0}^{n-[n^{1/4}]} P_o[\tau_x=m] P(r,n) + \sum_{m = n - [n^{1/4}] +1}^n P_o[\tau_x=m] \Biggr) \nonumber\\
&\leq  \size P(r,n)  + n^{1/4} . \nonumber
\end{align}
In order to bound $P(r,n)$, fix any $x \in V$. Denote the successive return times to $x$ by $(T_k)_{k \geq 1}$, that is, 
$T_1 = \tau_{x}^+ \text{, and for $k \geq 2$, }
T_k = \inf\{ n > T_{k-1}: X_n =x\}.$
Since any $y \in B(x,r)$ is within distance $r$ of $x$, we have $P_x[ \tau_y \leq \tau^+_x] \geq d^{-r}$ for any such $y$. With the strong Markov property applied at the times $T_k$, we infer that for any $m \geq 1$,
\begin{align}
\label{entub1.3}
P(r,n) &\leq \sup_{x \in V} \sum_{y \in B(x,r)} P_x[\tau_y > n^{1/4}]\\ 
&\leq \sup_{x \in V} \sum_{y \in B(x,r)}  \Big(P_x [\tau_y > T_m] + P_x[T_m > n^{1/4}] \Big) \nonumber\\
 &\leq \sup_{x \in V} |B(x,r)|  \Big((1-d^{-r})^m + m P_x[\tau_x^+ \geq n^{1/4}/m] \Big). \nonumber
\end{align}
Inserting the above estimates into \eqref{entub1}, we deduce that
\begin{align*}
\frac{\ent_o(R_n)}{\size} \leq  c \frac{2 \log s_r}{r} + cb_r^2 (1-d^{-r})^m + cb_r^2 m \sup_{x \in V} P_x \Bigl[ \tau^+_x > \frac{ n^{1/4}}{m} \Bigr] +  c_d b_r \frac{ n^{1/4}}{\size},
\end{align*}
for any $r,m,n \geq 1$ and $o \in V$. By assumption \eqref{e:unire} and Proposition~\ref{p:rlb}, applied to the last two terms, the supremum over $o \in V$ of right-hand side converges as $n$ tends to infinity to $(c (\log s_r)/r) +c b_r^2(1-c_{d,r})^m$. Letting $m$ tend to infinity, then using \eqref{e:vg}, we obtain \eqref{e:entub}. \end{proof}

\begin{figure}
\psfrag{a1}[cc][cc][1.5][0]{$a_1$}
\psfrag{a2}[cc][cc][1.5][0]{$a_2$}
\psfrag{o1}[cc][cc][1.5][0]{$o_1$}
\psfrag{o2}[cc][cc][1.5][0]{$o_2$}
\psfrag{o3}[cc][cc][1.5][0]{$o_3$}
\begin{center}
\includegraphics[angle=0,
width=0.7\textwidth]{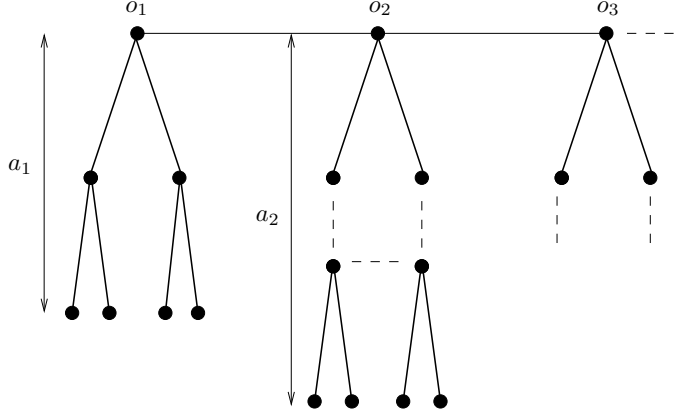}
\end{center}
\caption{The graph constructed in Remark~\ref{r:exrec}.} \label{fig:c}
\end{figure}
\begin{remark}
\label{r:exrec}
\textup{
The conclusion \eqref{e:entub} of Theorem~\ref{t:entub} does not hold for every recurrent graph with bounded degrees, not even without the supremum over $o \in V$. As a counterexample, consider an infinite sequence of binary trees of finite depths $a_1 < a_2 < \cdots $, with roots $o=o_1, o_2, \ldots$. Regardless of the choice of $a_1, a_2, \ldots$, after addition of edges $$\{o_1,o_2\}, \{o_2, o_3\}, \ldots$$ one obtains a connected and recurrent graph, illustrated in Figure~\ref{fig:c}. The sequence of depths can be chosen recursively such that 
\begin{align}
\label{exrec1}
P_{o}[\tau_{o_n} \geq \sqrt{a_n} -1] \leq 1/2 
\end{align}
(observe that the distribution of $\tau_{o_n}$ depends on $a_1, \ldots, a_{n-1}$, but not on $a_n$). Denote the set of vertices in the $n$-th tree at distance $l$ from $o_n$ by $S_l$, $l \geq 1$. Observe that by an elementary estimate on one-dimensional biased random walk, whenever the random walk reaches a previously unvisited level $S_l$ at some vertex $x \in S_l$, the probability that the random walk never returns to $S_l$ until time $a_n$ and thereby leaves $x$ in $\partial_i R_{a_n}$ is at least $1/3$ (see, for example, \cite{D05}, Chapter 4, Example~7.1, p.~271-272). Hence,
\begin{align*}
E_{o}[|\partial_i R_{a_n}|] &\geq \frac{1}{3} \sum_{l=0}^{a_n} P_{o}[\tau_{S_l} \leq a_n] \\
& \geq c \sum_{l=0}^{a_n} P_{o}[\tau_{S_1} \leq \sqrt{a_n}] P_{S_1}[\tau_{S_l} \leq a_n - \sqrt{a_n} ] \text{ (Markov prop.)}\\
&\geq c \sum_{l=0}^{a_n} P_{S_1} [\tau_{S_l} \leq a_n - \sqrt{a_n}] \text{ (by \eqref{exrec1})}\\
&\geq c E_{S_1}[d(o_n, X_{[a_n - \sqrt{a_n}] \wedge \tau_{o_n}})].
\end{align*}
Since $\big(d(o_n, X_{k \wedge \tau_{o_n}}) - (k \wedge \tau_{o_n})/3 \big)_{k \geq 0}$ is a martingale (see again \cite{D05}, p.~272), the right-hand side is bounded from below by
$$c E_{S_1}[(a_n-\sqrt{a_n}) \wedge \tau_{o_n}] \geq c (a_n - \sqrt{a_n}) P_{S_1} [\tau_{o_n} > a_n] \geq c' a_n \geq c' \langle R_{a_n} \rangle,$$
using again an elementary estimate on one-dimensional biased random walk for the second inequality. Hence, Corollary~\ref{c:bdry} shows that $\ent_o(R_n)/\size$ does not tend to $0$ for the recurrent graph $G$ defined above.
}
\end{remark}

\section{Vertex-transitive graphs} \label{s:dich}

This final section contains the proof of Theorem~\ref{t:dich} on the dichotomy for vertex-transitive graphs. Note that, due to vertex-transitivity, $\ent_o(R_n)$, and $\size$ do not depend on $o$, so $o$ can be omitted from the notation. We first deduce the estimate in the recurrent case.

\begin{corollary} \label{c:entub}
Let $G$ be any vertex-transitive and recurrent graph. Then
\begin{align}
\frac{\entt}{\sizet} \longrightarrow 0, \text{ as } n \to \infty.
\end{align}
\end{corollary}

\begin{proof}
By Theorem~\ref{t:entub}, we only need to check conditions \eqref{e:unire} and \eqref{e:vg}. Due to vertex-transitivity, the suprema become superfluous in both conditions. The assertion \eqref{e:unire} is thus a consequence of recurrence and monotone convergence. Simple random walk on a vertex-transitive graph is quasi-homogeneous, see \cite{W00}, Theorem~4.18, p.~47. As a consequence, if $\liminf_r |B(o,r)|/r^3>0$, then $G$ satisfies a three-dimensional isoperimetric inequality, which implies in particular that the random walk is transient (we refer to \cite{W00}, Corollary~4.16, p.~47, and Theorem~5.2, p.~49, for the proofs of these claims). Hence, we must have $\liminf_r |B(o,r)|/r^3=0$, from which \eqref{e:vg} immediately follows. 
\end{proof}

Theorem~\ref{t:dich} follows:

\begin{proof}[Proof of Theorem~\ref{t:dich}.]
Since the infimum in condition \eqref{e:ut} is superfluous for vertex-transitive graphs, every vertex-transitive and transient graph is uniformly transient. Theorem~\ref{t:entlb} hence shows that $$\liminf_{n \to \infty} \entt/\sizet > 0$$ in the transient case. Corollary~\ref{c:entub} shows that $\entt/\sizet$ tends to $0$ for vertex-transitive recurrent graphs. Since every vertex-transitive graph is either transient or recurrent, this proves Theorem~\ref{t:dich}.
\end{proof}

In order to answer the question from \cite{BKYY09} mentioned after Theorem~\ref{t:dich}, it remains to prove that $\sizet$ grows linearly in $n$ for vertex-transitive transient graphs. This is surely well-known, but we could not find a proof in the literature, so we give a proof here. We denote the Green function of the simple random walk (evaluated on the diagonal) by
\begin{align*}
g(x) = E_x \Biggl[ \sum_{k=0}^\infty \mathbf{1}_{\{X_k=x \}} \Biggr], \text{ for } x \in V,
\end{align*}
and the Green function of the random walk killed after $n$ steps by
\begin{align*}
g^n(x) = E_x \Biggl[ \sum_{k=0}^n \mathbf{1}_{\{X_k=x \}} \Biggr], \text{ for } x \in V.
\end{align*}

\begin{proposition} \label{p:Rllnt}
Let $G$ be any transient graph and let $$e(x) = P_x[\tau^+_x = \infty] \in (0,1)$$ be the escape probability from vertex $x \in V$. Then
\begin{align}
\label{e:Rllnt1}
\inf_{o \in V} e(o) \leq \liminf_{n \to \infty} \inf_{o \in V} \frac{\size}{n}.
\end{align}
Moreover, if $G$ is vertex-transitive and transient, then
\begin{align}
\label{e:Rllnt2} \frac{\sizet}{n} \longrightarrow e(o)>0, \text{ as } n \to \infty.
\end{align}
\end{proposition}

\begin{proof}
Let $G$ be any transient graph and $o \in V$. By the Markov property applied at time $1$, $g(x) = 1 + (1-e(x))g(x),$ hence
\begin{align}
\label{Rllnt1}
g(x) = 1/e(x), \text{ for any } x \in V.
\end{align}
For $n \geq 1$ and $x \in V$, we denote by $N^n_x$ the total number of visits to $x$ in the first $n$ steps, 
\begin{align*}
N^n_x = \sum_{0 \leq k \leq n} \mathbf{1}_{\{X_k=x\}}.
\end{align*}
Summation over all vertices yields the total number of steps:
\begin{align}
\label{Rllnt2}
n+1 = \sum_{x \in V} N^n_x.
\end{align}
Taking expectations in this equation and applying the strong Markov property at time $\tau_x$, we obtain that
\begin{align}
\label{Rllnt3}
n+1 &= \sum_{x \in V} E_o[N^n_x]\\
&\leq \sum_{x \in V} P_o[\tau_x \leq n] g(x) \nonumber\\
&\leq \size \sup_{x \in V} g(x). \nonumber
\end{align}
The estimate \eqref{e:Rllnt1} is trivial if $\inf_{o \in V}e(o)=0$ and follows from \eqref{Rllnt1} and \eqref{Rllnt3} otherwise. 

Let now $G$ be a vertex-transitive transient graph. In particular, \eqref{e:Rllnt1} holds without the infimum on either side. Replacing $n$ by $[\lambda n]$ with $\lambda > 1$ in \eqref{Rllnt2}, we have for $n \geq c_\lambda$,
\begin{align*}
[\lambda n] + 1 &\geq \sum_{x \in V} E_o[N^{[\lambda n]}_x \mathbf{1}_{\{\tau_x \leq n\}} ] \\
&\geq \sum_{x \in V} P_o[\tau_x \leq n] g^{[\lambda n]-n}(o)\\
&= \size  g^{[\lambda n]-n}(o).
\end{align*}
Letting $n$ tend to infinity and using \eqref{Rllnt1}, it follows that for any $\lambda >1$,
\begin{align*}
\limsup_{n \to \infty} \frac{\size}{n} \leq \lambda e(o).
\end{align*}
We now let $\lambda$ tend to $1$ and together with \eqref{e:Rllnt1} obtain \eqref{e:Rllnt2}.
\end{proof}

\begin{remark} \label{r:answer}
\textup{
Theorem~\ref{t:entlb} and \eqref{e:Rllnt1} prove that $\ent_o(R_n)$ grows linearly in $n$ for all uniformly transient graphs with bounded degrees, while Theorem~\ref{t:entub} in particular proves that $\ent_o(R_n)$ grows sublinearly in $n$ on uniformly recurrent graphs satisfying the growth condition~\eqref{e:vg}.
}
\end{remark}

\begin{remark} \label{r:ign}
\textup{
The example presented in Remark~\ref{r:exrec} does not satisfy the volume growth assumption \eqref{e:vg} and hence does not show that \eqref{e:vg} is actually necessary in Theorem~\ref{t:entub}. We do not know of an example showing necessity of \eqref{e:vg}, or indeed if there even exists a graph satisfying \eqref{e:unire} but not \eqref{e:vg}.
}
\end{remark}

\begin{remark}
\textup{
Given the results of the present work, it is natural to wonder whether one can obtain more precise estimates on $\ent_o(R_n)$ on vertex-transitive graphs or even prove an analogue of Shannon's theorem on the almost-sure behavior of $\log(p_o(R_n))$.
}
\end{remark}


\begin{thebibliography}{99}

\bibitem{BKYY09} I.~Benjamini, G.~Kozma, A.~Yadin, A.~Yehudayoff.
Entropy of random walk range. \emph{Ann.~Inst.~H.~Poincar\'e Probab.~Statist.}, to appear. 

\bibitem{CLR01} T.H.~Cormen, C.E.~Leiserson, R.L.~Rivest, C.~Stein. \emph{Introduction to algorithms}. (second edition) MIT Press, Cambridge, MA, 2001.

\bibitem{CT91} T.M.~Cover, J.A.~Thomas. \emph{Elements of information theory.} John Wiley \& Sons, Inc., New York, 1991.

\bibitem{D05} R.~Durrett. \emph{Probability: Theory and Examples.} (third edition) Brooks/Cole, Belmont, 2005.

\bibitem{E03} A.~Erschler. On drift and entropy growth for random walks on groups. \emph{Ann. Probab.} 31/3, pp.~1193-1204, 2003.

\bibitem{KV83} V.A.~Ka{\u\i}manovich,  A.M.~ Vershik. Random walks on discrete groups: boundary and entropy. \emph{Ann. Probab.}, 11/3, pp.~457-490, 1983.

\bibitem{NP08} A.~Naor, Y.~Peres. Embeddings of discrete groups and the speed of random walks. \emph{International Mathematics Research Notices} 2008, Art. ID rnn 076.

\bibitem{R05} P.~R\'ev\'esz. \emph{Random walk in random and non-random environments}. Second edition. World Scientific Publishing Co.~Pte.~Ltd., Singapore, 2005.

\bibitem{S48} C.E.~Shannon. A mathematical theory of communication. \emph{The Bell System Technical Journal}, Vol. 27, pp.~379-423, 623-656, 1948.

\bibitem{V85} N.T.~Varopoulos. Long range estimates for Markov chains. \emph{Bull. Sci. Math.} (2) 109, pp.~225-252, 1985.

\bibitem{W00} W.~Woess. \emph{Random walks on infinite graphs and groups.} Cambridge University Press, Cambridge, 2000.







\end{thebibliography}
\end{document}